\documentclass[11pt]{amsart}
\usepackage{pifont}
\usepackage{tikz}
\usepackage[utf8]{inputenc}
\usepackage[T1]{fontenc}
\usepackage{amsmath,amscd,amssymb}
\newtheorem{theorem}{Theorem}[section]

\newtheorem{lemma}[theorem]{Lemma}
\newtheorem{proposition}[theorem]{Proposition}
\newtheorem{remark}[theorem]{Remark}
\numberwithin{equation}{section}
\newtheorem{definition}[theorem]{Definition}

\newtheorem{corollary}[theorem]{Corollary}

\input xy
\xyoption{all}

\begin{document}
\title[Existence of Ulrich Bundle on general Surfaces with positive canonical bundle ]{Existence of Ulrich Bundle on some surfaces of general type}
\author{SURATNO BASU}
\email{suratno.b@srmap.edu.in, suratno.math@gmail.com}
\address{SRM University, Amaravathy,
         Andhra Pradesh, India}
\author{SARBESWAR PAL}
\email{sarbeswar11@gmail.com, spal@iisertvm.ac.in}
\address{IISER Thiruvananthapuram, Maruthamala P. O., Kerala 695551}


\begin{abstract}
 Let $X$ be a  smooth projective algebraic surface of Picard rank one with very ample canonical bundle $K_X$. We further assume that  $q +1 \le \chi(\mathcal{O}_X)$.  
 In this article, we will study the existence of an Ulrich bundle and its stability property of it with respect to $K_X$.
\end{abstract}

\maketitle

\section{Introduction}

Let $X \subset \mathbb{P}^n$ be a smooth projective variety. 
Then $X$ is
naturally polarised by the very ample line bundle $h:= \mathcal{O}_X(1):= \mathcal{O}_{\mathbb{P}^n}(1) \otimes \mathcal{O}_X$.\\
A vector bundle $E$ on $X$ is called Ulrich with respect to $\mathcal{O}_X(1)$, if 
 \[
  h^i(X, E(-i))= h^j(X, E(-(j+1))= 0,
 \]
for each $i > 0$ and $j < \text{dim}(X)$. The study of such bundles started in 80's  by Ulrich  \cite{UL} in Commutative Algebra context  with different name. 
The study of Ulrich bundles in Algebraic Geometry context was started by Eisenbud and Schreyer \cite{ES}.  
Recently, this has  become an active research  area in Algebraic Geometry. The study of Ulrich bundles is closely related to the question: 
whether X can be defined set-theoretically by
a linear determinant and also it has nice applications in  investigating various  geometric questions.  Noting the powerful applications of Ulrich bundles,  
the first fundamental question asked by Eisenbud and Schreyer \cite{ES}.\\

{\bf Question}:  Does every embedded variety $X \subset \mathbb{P}^n$ have an Ulrich sheaf? If X has an Ulrich
sheaf, what is the smallest possible rank for such a sheaf?\\

 When $X$ is a smooth projective curve the the above question has an easy affirmative answer. In fact, in this case, $X$ supports an Ulrich line bundle.\\
 When $X$ has dimension $> 1$, then there is no general result known of the  above question. 
 However, some partial results are known. 
 For interested reader we refer to \cite{AB1}, \cite{AB2}, \cite{AB3}, \cite{GC1}, \cite{GC2}, \cite{GC3}, \cite{GC4}, \cite{GC5}, \cite{AFO}, \cite{AR}, \cite{PP}
 and the references there in.\\
 In this article, we will study the question when $X$ is a smooth projective surface with some extra hypothesis. More precisely we prove the following:
 \begin{theorem}\label{T1}
 Let $X$ be a smooth projective surface over complex numbers with $\text{Pic}(X) = \mathbb{Z}[K_{_X}]$ where $K_X$ is the canonical bundle. Further assume that $K_X$ is very ample and 
 $q(X):= h^1(\mathcal{O}_X) < \chi(\mathcal{O}_{_X})$-1. Then $X$ supports a Ulrich bundle of rank $2$ and such a bundle is  necessarily stable. 
 \end{theorem}
 
 Note that under the hypothesis of Theorem \ref{T1} there does not exist any Ulrich line bundle on $X$ with respect to $K_X$.
 Suppose, $L=\mathcal{O}_{_X}(k)$. If $k\geq 1$ then $H^0(X,\mathcal{O}_{_X}(k-1))\neq 0$. If $k\leq 0$ then 
 $H^2(X,\mathcal{O}_{_X}(k-2))=H^0(X,\mathcal{O}_{_X}(-k+2+1))\neq 0$. Therefore, no line bundles on $X$ can be Ulrich with respect to $K_{_X}$.

Once that the existence of Ulrich bundles of low ranks is proved, one could be
interested in understanding how large a family of Ulrich bundles supported on $X$
can actually be. In particular we say that a smooth variety $X \subset \mathbb{P}^n$ is Ulrich–wild
if it supports families  of pairwise non–isomorphic, indecomposable,
Ulrich bundles of  arbitrary dimension.
Our next Theorem deals with the question of Ulrich wildness of the surface. More precisely, we prove the following:
\begin{theorem}\label{T2}
Let $X$ be a smooth projective surface over complex numbers which satisfies the hypothesis of Theorem \ref{T1}. Then  $X$ is Ulrich wild.
\end{theorem}
{\bf Sketch of the proof of Theorem \ref{T1}}: Our approach is very simple.  
It is known (Corollary \ref{c1}) that a rank two vector bundle $E$ on a surface $X$ with $\text{Pic}(X)\simeq \mathbb{Z}$ is Ulrich with respect to an ample line bundle $H:= \mathcal{O}_X(1)$, 
if and only if $c_1(E) = 3H + K_X$, where $K_X$ is the canonical bundle , $c_2(E)= \frac{5H^2 + 3H.K_X}{2} + 2 \chi(\mathcal{O}_X)$ and $E$ is initialized. \\
We first show that the moduli space of rank $2$ stable vector bundles with $c_1 = 3H +K_X$ and $c_2= \frac{5H^2 + 3H.K_X}{2}+ 2 \chi(\mathcal{O}_X)$ is non empty. 
Then we will show that a general element in the moduli space is initialized. 

{\bf Organization of the paper}: In Section 2, we will recall some basic facts which we will be using in subsequent sections. \\
In Section 3, we will prove the main Theorem \ref{T1}.\\
In Section we will prove the Ulrich wildness Theorem \ref{T2}.

{\bf Acknowledgement}: We are extremely thankful to Prof. Angelo Felice Lopez for all his comments and suggestions. We have majorly rewritten Lemma $4.1$ and Proposition $4.4$ and fixed certain gaps in the arguments therein following his remarks. We owe a proof of smoothness of certain curves in Lemma 4.1 to him.

\section{Notation} Throughout the paper we will work with the following notations:
\begin{itemize}
 \item All schemes $X$ will be defined over the complex numbers $\mathbb{C}$.
 \item For any projective scheme $X$ by $\chi$ we mean the Euler characteristic $\chi(\mathcal{O}_{_X})$.
 \item For any closed subscheme $Z$ of $X$, $I_{Z,X}$ will denote the ideal sheaf of $Z$ in $X$.
 \item For any sheaf $F$ of $\mathcal{O}_{_X}$-module we write $H^i(F):=H^i(X,F)$ and $h^i(F):=\text{dim}H^i(X,F)$, $i\geq 0$.
 \item For any vector bundle $E$ on $X$, $E^{*}$ will denote the dual of $E$. For any vector bundle $E$ on $X$, $\mu(E)$ denotes it slope with respect 
       to some ample line bundle.
\end{itemize}

\section{Preliminaries}
Let $X \subset \mathbb{P}^n$ be a smooth projective variety. 
Then $X$ is
naturally polarised by the very ample line bundle $h:= \mathcal{O}_X(1):= \mathcal{O}_{\mathbb{P}^n}(1) \otimes \mathcal{O}_X$.
\begin{definition}
 A vector bundle $E$ on $X$ is called Ulrich with respect to $\mathcal{O}_X(1)$, if 
 \[
  h^i(X, E(-i))= h^j(X, E(-(j+1))= 0,
 \]
for each $i > 0$ and $j < \text{dim}(X)$.
\end{definition}

\begin{proposition}\label{p1}
  Let $X$ be a surface endowed with a very ample line bundle $\mathcal{O}_X(1)$.
If $E$ is a vector bundle on $X$, then the following assertions are equivalent:\\
(a) $E$ is Ulrich with respect to $\mathcal{O}_X(1)$;\\
(b) $E^*(K_X +3h)$ is Ulrich with respect to $\mathcal{O}_X(1)$;\\
(c) $E$ is ACM and 

\begin{equation}\label{e1}
\begin{split}
c_1(E)h= \frac{\text{rank}(E)}{2}(3h^2 + K_Xh),\\
 c_2(E)= \frac{c_{_1}(E)^2 - c_1(E).K_X}{2} - \text{rank}(E)(h^2 -\chi(\mathcal{O}_X)).
 \end{split}
\end{equation}

(d) $h^0(E(-h))= h^0(E^*(2h+K_X))= 0$ and the equalities \ref{e1} hold.

\end{proposition}
\begin{proof}
See  \cite[Proposition 2.1]{GC1}.
\end{proof}
\begin{corollary}\label{c1}
  Let $X$ be a surface endowed with a very ample line bundle $\mathcal{O}_X(1)$.
If $E$ is a vector bundle of rank $2$ on $X$, then the following assertions are equivalent:\\
(a) $E$ is a special Ulrich bundle with respect to $\mathcal{O}_X(1)$;\\
(b) $E$ is initialized and 
\[
 c_1(E) = 3h+K_X, c_2(E)= \frac{5h^2 + 3K_Xh}{2} +2 \chi(\mathcal{O}_X).
\]

\end{corollary}
\begin{proof}
 See \cite[Corollary 2.2]{GC1}.
\end{proof}

\begin{remark}
 Under the hypothesis of the Theorem \ref{T1} if Ulrich bundle $E$ exsists then it is stable. 
 If $E$ is not stable then there exists a line subbundle $L$ with $c_1(L)=kh$, $k\geq 2$. 
 Thus $\mathcal{O}_{_X}(m)$, $m\geq 1$, is a subbundle of $E(-1)$. Therefore, $H^0(E(-1))\neq 0$ and hence $E$ is not Ulrich.
\end{remark}

\begin{definition}
A smooth variety $X \subset \mathbb{P}^n$  is Ulrich–wild
if it supports families  of pairwise non–isomorphic, indecomposable,
Ulrich bundles of  arbitrary dimension.
\end{definition}
\begin{theorem}\label{T3}
 Let $X$ be a smooth variety endowed with a very ample line bundle $\mathcal{O}_X(1)$.
If $A$ and $B$ are simple Ulrich bundles on $X$ such that $h^1(A \otimes B^*) \ge 3$,
and every non–zero morphism $A \to B$  is an isomorphism, then $X$ is Ulrich–wild.
\end{theorem}
\begin{proof}
 See \cite{FP}[Theorem A, Corollary 2.1 and Remark 1.6 (iii)].
 \end{proof}
 \subsection{Cayley-Bacharach Property}

  We recall that a locally complete intersection subscheme $Z$
of dimension zero on a surface $X$ is said to have Cayley–Bacharach property (CB for short) with respect to
a line bundle $L$ if, for each subscheme  $Z^{'} \subset Z$  of co-length $1$, 
 the natural morphism $H^0(I_{Z,X}(L)) \to H^0(I_{{Z,X}^{'}}(L))$
is an isomorphism.
\begin{theorem}\label{T4}
Let $C$ be a smooth curve and $f : C \to \mathbb{P}^r$  a morphism. Then any reduced
divisor $Z \in  |f^*(\mathcal{O}_{\mathbb{P}^r}(1))|$
 satisfies the Cayley-Bacharach condition with respect to $K_C$, where $K_C$ denotes the canonical bundle of $C$.
\end{theorem}
\begin{proof}
See \cite[Theorem 2.3]{HLU}.
\end{proof}
\begin{corollary}\label{C1}
Let $X$ be a smooth projective surface with $\text{Pic}(X) =\mathbb{Z} \text{ and } C \in |\mathcal{O}_X(m)|$ be a smooth curve. 
Let $f : C \to \mathbb{P}^r$  be a morphism, and let $\Gamma \in |f^*(\mathcal{O}_{\mathbb{P}^r}(1))|$   
be general (and thus reduced). 
Then $\Gamma$ satisfies CB with respect to $K_X \otimes \mathcal{O}_X(m)$.
\end{corollary}
\begin{proof}
The corollary follows from the Theorem \ref{T4} and the  adjunction formula. 
\end{proof}
\begin{theorem}\label{T5}
Let $X$ be a surface and $Z \subset X$  a locally complete intersection subscheme of dimension $0$.
Then there exists a vector bundle $E$ of rank $2$ on $X$ fitting into an exact sequence
of the form
\[
0 \to \mathcal{O}_X \to E \to I_{Z,X}(L) \to 0
\]
 if and only if $Z$ has CB with respect to $L \otimes K_X$.
\end{theorem}
\begin{proof}
 See \cite[Theorem 5.1.1]{HL}.
 \end{proof}
 \subsection{Uniform position theorem}
 Let $C \subset \mathbb{P}^r, r \ge 3$ be an irreducible, non-degenerate curve of degree $d$.  Let $U \subset (\mathbb{P}^r)^*$ be the open subset of hyperplanes transverse to $C$. Let $\mathcal{D}$ be a linear system on $C$. Let $m$ be the maximal number of independent conditions imposed by a hyperplane section $C\cap H, H \in U$ on $\mathcal{D}$. \\
 Let $U^{'}:= \{H \in U: C \cap H \text{ impose at most } m-1 \text{ independent conditions on } \mathcal{D}\}$. Then $U^{'}$ is a proper subvariety of $U$. An element in the complement of $U^{'}$ in $U$ will be called a general hyperplane with respect to $\mathcal{D}$.
 \begin{theorem}[Uniform Position Theorem]
  Let $C \subset \mathbb{P}^r, r \ge 3$ be an irreducible, non-degenerate, possible singular curve of degree $d$. If $\mathcal{D}$ is any linear system on $C$, and $\Gamma= H \cap C$, a hyperplane section general with respect to $\mathcal{D}$, then all subsets of $k$ points of  $\Gamma$ impose the same number of conditions on $\mathcal{D}$.
  \end{theorem}
 
 \section{Existence of Ulrich Bundle}
\begin{lemma}\label{L1}
 Let $X$ be a surface with $\text{Pic}(X)=\mathbb{Z}[\mathcal{O}_{_X}(1)]$ and $K_{_X}=\mathcal{O}_{_X}(1)$ is very ample. Further assume that $q:= h^1(\mathcal{O}_X) < \chi -1$. Then there exists a zero dimensional 
 subscheme $\Gamma \subset X$ of length $4h^2+2\chi$ which satisfies Caley-Bachrach property with respect to the linear system $|\mathcal{O}_{_X}(5)|$. 
 \begin{proof}
  In view of Corollary \ref{C1}, we only need to  show that there exists a smooth curve $C\in |\mathcal{O}_{_X}(4)|$ and a base point free line bundle $L$ of degree $4h^2+2\chi$
  on $C$ with at least $2$ sections. 
  
  {\bf Steps to find line bundle $L$}:  For any curve $C\in |\mathcal{O}_{_X}(\alpha)|$ and a zero dimensional closed subscheme $B \subset C$ 
  we have the exact sequence 
  \[0\to I_{C, X}\to I_{B, X}\to I_{B, C}\to 0.\] Take $\alpha=4$ and $C\in |\mathcal{O}_{_X}(4)|$ an irreducible smooth curve. Then we have 
  
  \begin{equation}\label{1}
   0\to \mathcal{O}_{X}(-4)\to I_{B,X}\to I_{B, C}\to 0
  \end{equation}

  After tensoring the above exact sequence by $\mathcal{O}_{_X}(2)$ we get 
  \begin{equation}\label{2}
   0\to \mathcal{O}_{_X}(-2)\to I_{B,X}(2)\to \mathcal{O}_{_X}(2)|_{_C}\otimes \mathcal{O}_{_C}(-B)\to 0
  \end{equation}
  
  Note that $\text{deg}(\mathcal{O}_{_X}(2)|_{_C}\otimes \mathcal{O}_{_C}(-B))=8h^2-\text{length}(B)$.
  
  Our aim is now to find an irreducible curve $C\in |\mathcal{O}_{_X}(4)|$ and a closed subscheme $B\subset C$ of length $4h^2-2\chi$ such that the line 
  bundle $L:=\mathcal{O}_{_X}(2)|_{_C}\otimes \mathcal{O}_{_C}(-B)$ has at least $2$ sections and it is base point free.
  
  Take a smooth curve $C_0\in |\mathcal{O}_{_X}(2)|$. 
 Let  $Z_1:= H_1\cap C_0$ be a general hyperplane section with respect to the linear system $\mathcal{O}_X(4)_{\mid_{C_0}}$ (see section 3.2 for the definition of general hyperplane section with respect to a linear system). Then choose a hyperplane $H_2$ general with respect to the linear systems $\mathcal{O}_X(4)_{\mid_{C_0}}(-Z'_1)$, where $Z'_1$ is any subset in $Z_1$. Since there are only finitely many such subsets, such a choice of $H_2$ is always possible. Now take $Z_2$ as $C_0 \cap H_2$. 
 Note that we can make a choice of such $H_2$ such that $Z_1 \cap Z_2 = \varnothing$ 
 (that is always possible as a general $(n-2)$-plane do not intersect $C_0$).  Let $Z:= Z_1 \cup Z_2$. Then  $|Z|=4h^2$.\\
  We claim that there is a  subscheme $Z^{'} \subset Z$, of length $4h^2 -2 \chi$ such that for any  point 
  $P \in Z \setminus Z^{'}, h^0(X, I_{Z^{'}, X}(4)) > h^0(I_{Z^{'} \cup P, X}(4))$.\\
  {\bf Proof of the claim}:\\
  It is enough to show that there is subscheme $Z^{'} \subset Z$, of length $4h^2 -2 \chi$ such that for any  point $P \in Z \setminus Z^{'}, Z^{'}\cup P$ imposes independent conditions on $\mathcal{O}_X(4)$.\\
  Consider the exact sequence,
  \[
  0\to \mathcal{O}_{X}(2) \to I_{Z_1, X}(4) \to \mathcal{O}_X(4){\mid_{C_0}} \otimes \mathcal{O}_{C_0}(-Z_1) \to 0.
  \]
  Note that the sheaf $\mathcal{O}_X(4){\mid_{C_0}} \otimes \mathcal{O}_{C_0}(-Z_1)$ in the above exact sequence is isomorphic to $\mathcal{O}_X(3)\mid_{C_0}$. On the other hand, since the line bundle $\mathcal{O}_X(1)$ is ample, by Mumford's theorem $h^1(\mathcal{O}_X(-1)) =h^1(\mathcal{O}_X(2)) = 0$.
  Thus we have, $h^0(X, I_{Z_1, X}(4)) = h^0(X, \mathcal{O}_X(2)) + h^0(X, \mathcal{O}_X(3)\mid_{C_0})$. Now  from the natural exact sequence,
  \[
   0 \to \mathcal{O}_X(1) \to  \mathcal{O}_X(3) \to \mathcal{O}_X(3)\mid_{C_0} \to 0,
  \]
  we have $h^0(X, \mathcal{O}_X(3)\mid_{C_0}) = h^0(X, \mathcal{O}_X(3)) -h^0(X, \mathcal{O}_X(1)) + q$. 
  Thus we have 
  \[
  \begin{split}
  h^0(X, I_{Z_1, X}(4)) = h^0(X, \mathcal{O}_X(3)) + h^0(X, \mathcal{O}_X(2)) -h^0(X, \mathcal{O}_X(1)) + q  & \\
  = 3h^2 + \chi + h^2 + \chi - \chi +1 -q +q=4h^2 + \chi +1.
  \end{split}
  \]
  
  In other words, $Z_1$ imposes  $6h^2 + \chi - 4h^2 - \chi -1 =2 h^2 - 1$ independent conditions on sections of $\mathcal{O}_X(4)$. 
  Similarly, $Z_2$ imposes  $2h^2 -1$ independent conditions on sections of $\mathcal{O}_X(4)$.  
  Considering $Z_1$ and $Z_2$ together one can show that $Z$ imposes at least $4h^2 -\chi +1 - q$ independent conditions on sections of $\mathcal{O}_X(4)$.\\
  Since $q \le \chi -1, 4h^2 - 2\chi +2 \le 4h^2 - \chi +1 -q$,  there exists a set of  $4h^2 - 2\chi +2$ points in $Z$ which imposes independent 
  conditions on sections of $\mathcal{O}_X(4)$.\\
  Note that if a zero dimensional subscheme, $Y$ imposes $m$ conditions on $\mathcal{O}_X(d)$ and a  subset $Y_1$  of $Y$ imposes $|Y_1|$ conditions on $\mathcal{O}_X(d)$, then one can always extend $Y_1$ to a subset $Y_2$ of $Y$ of length $m$, such that $Y_2$ imposes independent conditions on $\mathcal{O}_X(d)$.\\
  Therefore, we can find a subset
  $Z^{'}$ of $Z$ of  $4h^2  -2\chi$ points intersecting both $Z_1$ and 
  $Z_2$ in at most  $2h^2 - 2$ points and imposes independent conditions on sections of $\mathcal{O}_X(4)$  in the following way:\\
  Choose a subset $\tilde{Z_1}$ of  $2h^2 -1$ points from $Z_1$ imposing independent conditions. Then by above observation, we can always extend $\tilde{Z_1}$ to a set $W$ of length $4h^2 -2\chi +2$  which imposes independent conditions. Note that the length of $W \cap Z_2$ is equal to $2h^2 -2\chi +3$. Since $q + 1 \le \chi,$ for any $q$ one can see that $\chi \ge 2$. Thus $2h^2-2\chi +3 \le 2h^2 -1$.\\
   Now take  $Z^{'}$ as $(W \cap Z_2) \cup \tilde{Z_1} \setminus \{\tilde{P}, \tilde{P}^{'}\}$ for some point $\tilde{P} \in \tilde{Z_1}$ and $\tilde{P}^{'} \in W \cap Z_2$.\\
  Let $P \in Z \setminus Z^{'}$ be an arbitrary point.\\ 
  {\bf Case (1)}: $P \in Z_1$. \\
  Let $Z_i^{'}:= Z_i \cap Z^{'}, i= 1, 2$. 
  Since $Z_1$ is a general hyperplane section of $C_0$ with respect to the linear systems  $|\mathcal{O}_X(4){\mid_{C_0}}|$, any point $Q \in Z_1 \setminus Z_1^{'}, Z_1^{'} \cup \{Q\}$ imposes 
   same number of conditions  on the linear series $|\mathcal{O}_X(4){\mid_{C_0}}|$. Thus the linear systems
   $|\mathcal{O}_X(4){\mid_{C_0}}  \otimes \mathcal{O}_{C_0}(-Z_1^{'}-\tilde{P})|$ and $|\mathcal{O}_X(4){\mid_{C_0}}  \otimes \mathcal{O}_{C_0}(-Z_1^{'}-P)|$ have the same dimension. Again since $Z_2$ is a general hyperplane section with respect to the linear systems $|\mathcal{O}_X(4){\mid_{C_0}}  \otimes \mathcal{O}_{C_0}(-Z_1^{'}-\tilde{P})|$ and $|\mathcal{O}_X(4){\mid_{C_0}}  \otimes \mathcal{O}_{C_0}(-Z_1^{'}-P)|,  |\mathcal{O}_X(4){\mid_{C_0}} \otimes \mathcal{O}_{C_0}(-Z_2^{'}) \otimes \mathcal{O}_{C_0}(-Z_1^{'}-\tilde{P})|$ and $|\mathcal{O}_X(4){\mid_{C_0}} \otimes \mathcal{O}_{C_0}(-Z_2^{'}) \otimes \mathcal{O}_{C_0}(-Z_1^{'}-P)|$ have same dimension. 
   Note that $Z^{'} \cup \{\tilde{P}\}$ imposes independent conditions on sections of $\mathcal{O}_X(4)$. Thus from the exact sequence 
   \[
  0\to \mathcal{O}_{X}(2) \to I_{Z^{'} \cup \{\tilde{P}, X\}}(4) \to \mathcal{O}_X(4){\mid_{C_0}} \otimes \mathcal{O}_{C_0}(-Z_2^{'}) \otimes \mathcal{O}_{C_0}(-Z_1^{'}-\tilde{P}) \to 0.
  \] 
  it follows that $Z^{'} \cup P$ imposes independent conditions on sections of $\mathcal{O}_X(4)$.\\
  {\bf Case (2)}: $P \in Z_2$.
  Since $Z_2$ is a general hyperplane section of $C_0$ with respect to the linear systems  $|\mathcal{O}_X(4){\mid_{C_0}} \otimes \mathcal{O}_{C_0}(-Z_1^{'})|$,  by Uniform Position Theorem, for any point $Q \in Z_2 \setminus Z_2^{'}, Z_2^{'} \cup \{Q\}$ imposes 
   same number of conditions  on the linear series $|\mathcal{O}_X(4){\mid_{C_0}} \otimes \mathcal{O}_{C_0}(-Z_1^{'})|$. \\
  Again $Z^{'} \cup \{\tilde{P}^{'}\}$ imposes independent conditions on $\mathcal{O}_X(4)$. Thus as earlier case from the exact sequence
  \[
  0\to \mathcal{O}_{X}(2) \to I_{Z^{'} \cup \{\tilde{P}^{'}\}, X}(4) \to \mathcal{O}_X(4){\mid_{C_0}} \otimes \mathcal{O}_{C_0}(-Z_2^{'}-\tilde{P}^{'}) \otimes \mathcal{O}_{C_0}(-Z_1^{'}) \to 0.
  \] 
  one can conclude the claim.
  
  
  Let $Z^{'} \subset Z$ be a subscheme of length $4h^2 - 2\chi$ which imposes independent condition on the linear system $|\mathcal{O}_{_X}(4)|$. We claim that a general element 
  $C\in |I_{Z',X}(4)|$ is smooth. To see this first we observe that, as $Z'$ imposes independent condition on $|\mathcal{O}_{_X}(4)|$, $h^1(I_{Z',X}(4))=0$. Also it is easy to see that $h^2(I_{Z',X}(3))=0$, use the exact sequence 
  \[0\to I_{Z',X}(3)\to \mathcal{O}_{X}(3)\to \mathcal{O}_X(3)|_{Z'}\to 0.\] Thus $I_{Z',X}(4)$ is $0$-regular in the sense of Castelnuovo-Mumford. Hence it is 
  globally generated. Blow up the surface $X$ at the points of 
  $Z'$. Then $f^*\mathcal{O}_{_X}(4)-\sum_i{_{=1}^{4h^2-2\chi}}E_i$ is globally generated where $f$ is the blow up morphism and $E_i$ are the exceptional divisor. We have $(f^*O_X(4)-E_1-...-E_k)^2 = 12h^2+2\chi > 0$. It follows from the Bertini theorem that the general members of the linear system $|f^*O_X(4)-E_1-...-E_k|$ are smooth. Pick one such curve and denote it by $\tilde{C}$. Then $C.E_i=1$ for all $i$. Hence $f|\tilde{C}:\tilde{C}\to X$ is an embedding. We denote the image of this curve by $C$. Thus we get a smooth curve $C\in |\mathcal{O}_{_X}(4)|$ passing through the points of $Z'$.
   Once we know the above fact by previous step we choose $C$, 
  such that we have $C \cap C_0 \cap H_1H_2 = Z^{'}$.\\
  Since $h^0(X, I_{Z^{'}, X}(2) \ge 2$, taking $B = Z^{'}$ in \ref{2}, we have a line bundle, namely, 
  $L:= \mathcal{O}_{_X}(2)|_{_C}\otimes \mathcal{O}_{_C}(-Z^{'})$ on $C$ which admits at least two sections. \\
  Take $D_1=C_0.C-Z^{'}$ and $D_2=H_1H_2.C-Z^{'}$. Then $D_1$ and $D_2$ are linearly independent effective divisors lying in the linear system 
  $|\mathcal{O}_{_X}(2)|_{_C}\otimes \mathcal{O}_{_C}(-Z^{'})|$. They are base point free since if $q$ is a base point then $q$ is in the support of 
  $C_0\cap H_1H_2 \cap C-Z^{'}$. But $C_0 \cap H_1H_2 \cap C-Z^{'}=0$. 
  
  Thus $|L|=|\mathcal{O}_{_X}(2)|_{_C}\otimes \mathcal{O}_{_C}(-Z^{'})|$ is a base point free linear system of positive dimension which proves the Lemma.
 \end{proof}

\end{lemma}

\begin{corollary}\label{C2}
 Under the hypothesis of Lemma \ref{L1} there exists a rank $2$ locally free sheaf $E$ which fits into the exact sequence 
 \[0\to \mathcal{O}_{_X} \to E\to I_{\Gamma, X}(4)\to 0\] where $\Gamma$ is a $0$ dimensional closed subscheme 
 of length $4h^2+2\chi$.
\begin{proof}
 By Theorem \ref{T5}, there exists a rank $2$ locally free sheaf fitting into the exact sequence 
 \[0\to \mathcal{O}_{_X} \to E\to I_{\Gamma, X}(4)\to 0\] if and only if there exists $0$-dimensional subscheme 
 $\Gamma \subset X$ satisfying Caley-Bacharach property with respect to $|\mathcal{O}_{_X}(4)\otimes K_{_X}|=|\mathcal{O}_{_X}(5)|$.
 By Lemma \ref{L1} there exists a $0$-dimensional subscheme $\Gamma \subset X$ of length $4h^2+2\chi$ which satisfies the Caley-Bacharach property with respect to the 
 linear system $|\mathcal{O}_{_X}(5)|$. Thus we are done. 
 
\end{proof}

\end{corollary}
\begin{proposition}\label{p2}
A general rank $2$ vector bundle appearing in the exact sequence of Corollary \ref{C2} is stable.
\end{proposition}
\begin{proof}
 Let $\sum(X,4h^2+2\chi)$ denote the moduli variety of rank $2$ locally free extensions of $\mathcal{O}_{_X}$ by $I_{X,\Gamma}(4)$ where 
 $\Gamma\in Hilb^{4h^2+2\chi}(X)$. By Corollary \ref{C2} this variety is non empty. We now show that the general element of this variety is stable.
 
 Let $E\in \sum(X,4h^2+2\chi)$. Let $L$ be a proper subundle of  $E$ i.e., $E/L$ is torsion free and $\text{rank}(L)=1$. 
 If $L\subset \mathcal{O}_{_X}$ then $\mu(L)<\mu(E)$. Thus if $E$ is not stable then the  composition map  $L\to E \to I_{\Gamma, X}(4)$ is non-zero. 
  This would immediately imply $h^0(\mathcal{O}_{_X}(4h-c_1(L))\otimes I_{\Gamma})>0$ and hence, $4h-c_1(L)$ is an 
 effective divisor. As $\text{Pic}(X)=\mathbb{Z}[\mathcal{O}_{_X}(1)]$ we have $L=\mathcal{O}_{_X}(k)$, $k\in \mathbb{Z}$. Since, $4h-c_1(L)$ is effective we have $k\leq 4$. Clearly,
 $k\neq 4$ otherwise, $h^0(I_{\Gamma, X})>0$ which is impossible. So, if $E$ is not stable we get $k=2$ or $3$. 
 
 We will now determine maximum possible dimension of the space of those bundles appearing in the exact sequence \ref{C2} and have $\mathcal{O}_{_X}(k)$ as subbundles where 
 $k=2,3$. 
 
 Case 1: Let $k=2$. Then we have an exact sequence 
 \begin{equation}\label{C3}
 0\to \mathcal{O}_{_X}(2)\to E\to I_{Z, X}(2)\to 0.
 \end{equation}
 We have $c_2(E)=4h^2+2\chi$ and from the above exact sequence we get 
 $c_2(E)=4h^2+\text{length}(Z)$. Therefore, $\text{length}(Z)=2\chi$ and the Hilbert scheme $\text{Hilb}^{2\chi}(X)$ of length $2\chi$ subschemes has dimension 
 $4\chi$. On the other hand the set of all extensions of the form \ref{C3} is parametrized by 
 $\text{Ext}^1(I_{Z, X}(2),\mathcal{O}_{_X}(2))\simeq H^1(I_{Z, X}(1))^{*}$ (by Serre duality). 
 Now 
 \[h^1(I_{Z, X}(1))=c+h^0(\mathcal{O}_{_X}(1)|_{_Z})-h^0(\mathcal{O}_{_X}(1))+h^1(\mathcal{O}_{_X}(1).\] where $c:=h^0(I_{Z, X}(1)$.
 As $h^0(\mathcal{O}_{_X}(1)|_{_Z})=2\chi$ and $h^1(\mathcal{O}_{_X}(1))=q$ we have $h^1(I_{Z, X}(1))=\chi+c+1$.
 
 Therefore, the dimension of the space of all extensions of the 
 form \ref{C3} when $Z\in \text{Hilb}^{2\chi}(X)$, is at most $4\chi+\chi+c+1=5\chi +c+1$. 
On the other hand, by \cite[corollary 3.5]{SIM1}, every component of  vector bundles appearing in the exact 
 sequence \ref{C2}  has dimension 
at least, $3(4h^2 + 2 \chi) -h^0(X, \mathcal{O}_X(5)) -1 = 12h^2 + 6\chi -10h^2 -\chi -1= 2h^2 +5\chi -1$.  \\
If $2h^2 +5\chi -1>5\chi+c+1, i.e., 2h^2 > c+2$ then it follows the general vector bundles appearing in \ref{C2} can not have $\mathcal{O}_{_X}(2)$ as subbundles. We will show the above inequality.
 From the natural exact sequence,
\[
0 \to \mathcal{O}_X \to \mathcal{O}_X(1) \to \mathcal{O}_H(1) \to 0,
\]
where $H$ is a general hyperplane section, one has, $h^2 \ge 2\chi + 2q -6$. 
 Thus $2h^2\geq 4(\chi+q-3)$.  On the other hand, since any two points impose independent conditions on sections of $\mathcal{O}_X(1), c \le h^0(\mathcal{O}_X(1))= \chi + q -1-2=\chi +q -3$. Thus $2h^2 -c-2 
 \ge 4(\chi+q-3)-(\chi+q-3) -2=3\chi+3q-9$. As $\mathcal{O}_{X}(1)$ is very ample $h^0(\mathcal{O}_X(1)=\chi+q-1\geq 3$. Therefore,
 $\chi+q\geq 4$. Hence, $3(\chi+q)-9>0$.  
 Therefore, the general bundles appearing in \ref{C2} can not have $\mathcal{O}_{_X}(2)$ as subbundles. 
 
 Case 2: Let $k=3$. Then $\mathcal{O}_{_X}(3)\subset I_{X,\Gamma}(4)$. This would immediately imply $h^0(I_{X,\Gamma}(1))\neq 0$. Therefore, 
 $\Gamma \subset X\cap H$ for some hyperplane $H$. Now we will estimate the dimension of the space of all bundles appearing in the exact sequence 
 \ref{C2} and having $\mathcal{O}_{_X}(3)$ as subbundle. Let $S$ be the subscheme of $\sum(X,4h^2+2\chi)$ consists of those extensions 
 \[0\to \mathcal{O}_{_X}\to E\to I_{\Gamma, X}(4)\to 0,\] where $\Gamma \subset C$, $C\in |\mathcal{O}_{_X}(1)|$ and having $\mathcal{O}_X(3)$ as subbundle. 
 
 {\bf Claim}: $\text{dim}(S)\leq \frac{3h^2}{2}+3\chi-2+q$.\\
Let $\Gamma$ be a $0$ dimensional subscheme of length $4h^2+2\chi$. 
 Then the exact sequence of the form 
 \[0\to \mathcal{O}_{_X}\to E\to I_{\Gamma,X}(4)\to 0,\] are parametrised by 
 $\text{Ext}^1(I_{\Gamma,X}(4),\mathcal{O}_{_X})\simeq H^1(I_{\Gamma,X}(5))$. 
 Let $C\in |\mathcal{O}_{_X}(1)|$. 
 Then, as $\text{Pic}(X)\simeq \mathbb{Z}[\mathcal{O}_{_X}(1)$, it follows quite easily $C$ is an integral curve. Now we have an exact sequence 
 \begin{equation}\label{EZ}
 0\to I_{C,X}\to I_{\Gamma, X}\to I_{\Gamma, C}\to 0,
 \end{equation}
  of $\mathcal{O}_{_X}$ modules.
 Tensoring by $\mathcal{O}_{_X}(5)$ we get 
 \[0\to \mathcal{O}_{_X}(4)\to I_{\Gamma, X}(5)\to \mathcal{O}_{_X}(5)\otimes I_{\Gamma, C} \to 0,\]
 Now $\mathcal{O}_{_X}(5)\otimes I_{\Gamma, C}=I_{\Gamma, C}L$ where $L=\mathcal{O}_{_X}(5)|_{_C}$.
 The degree of $I_{C,\Gamma}L$  \[=5h^2-4h^2-2\chi=h^2-2\chi.\]
 and the degree of $K_{_C}$ is $2h^2$. Therefore, by the \cite[Theorem B]{ME}, we get 
 \[h^0(I_{\Gamma, C}L)\leq h^2/2-\chi+1.\]

 Thus $h^0(I_{\Gamma, X}(5))\leq h^0(\mathcal{O}_{_X}(4))+h^2/2-\chi+1$. This implies 
 $h^0(I_{X,\Gamma}(5))\leq 6h^2+\chi+h^2/2-\chi+1=6h^2+h^2/2+1$.
 
 On the other hand we have 
 \[0\to I_{\Gamma, X}\to \mathcal{O}_{_X}\to \mathcal{O}_{\Gamma}\to 0,\]
 Tensoring by $\mathcal{O}_{_X}(5)$ we get 
 \[0\to I_{\Gamma, X}(5)\to \mathcal{O}_{_X}(5)\to \mathcal{O}_{_X}(5)|_{_\Gamma}\to 0\]
 
 Therefore, $h^1(I_{\Gamma, X}(5))=h^0(\mathcal{O}_{_X}|_{\Gamma})+h^0(I_{\Gamma, X}(5))-h^0(\mathcal{O}_{_X}(5))$.
 This implies \[h^1(I_{\Gamma,X}(5))\leq 4h^2+2\chi+6h^2+h^2/2+1-10h^2-\chi=h^2/2+\chi+1.\]
 
 Let us count the dimension of the space of choices of   $\Gamma$ corresponding to  bundles in $S$.  There is a $\chi -1 +q$ dimensional
space of choices of the hyperplane sections $H$, and for each one we have an $4h^2 + 2\chi$ dimensional space
of choices of the subscheme $\Gamma$  of $4h^2 + 2\chi$ points in $H$. This gives the total dimension of such subschemes $\Gamma$ is $4h^2 + 2\chi + \chi -1+q= 4h^2 +3\chi -1+q$. \\
Thus the dimension of the space of all such extensions is $\le 4h^2 +3\chi -1+q + \frac{h^2}{2}+\chi +1-1=\frac{9h^2}{2} + 4\chi +q-1$.
 
 On the other hand, tensoring \ref{EZ}, by $\mathcal{O}_X(4)$, we have $h^0(X, I_{\Gamma,X}(4)) = 3h^2 + \chi$.
Thus $h^0(E) \ge 3h^2 + \chi +1 -q$. This means that
for a given bundle E, the space of choices of sections $s$ (modulo scaling) leading
to the subscheme of  $\Gamma \in \text{Hilb}^{4h^2 + 2\chi}(X)$, has dimension at least $ 3h^2 + \chi +1 - q -1= 3h^2 + \chi -q$. Hence the dimension of the space of bundles obtained by this construction is $3h^2 + \chi -q$ less than the dimension of the space of all extensions.

 Thus the maximum possible dimension of 
 $S$ is $ \frac{9h^2}{2} + 4\chi +q -1 -3h^2 -\chi +q=\frac{3h^2}{2} +3\chi +2q-1$.\\
 Again by  \cite[corollary 3.5]{SIM1}, every component of  $\sum(X,4h^2+2\chi)$ has dimension at least $2h^2+5\chi-1$.
 Since, under the assumption $q<\chi-1$,  $\frac{3h^2}{2}+3\chi-1+2q< 2h^2+5\chi-1$ we have $\text{dim}(S)<\text{dim}(\sum(X,4h^2+2\chi))$. Therefore, the general elements 
 of $\sum(X,4h^2+2\chi)$ are stable.



\end{proof}
Let $\mathcal{M}^s(2, 2h, h^2+2\chi)$ be the moduli space of stable vector bundle of rank $2$ with first Chern class $2h$ and second Chern class $h^2 + 2\chi$. 
By Proposition \ref{p2}, $\mathcal{M}^s(2, 2h, h^2+2\chi) \ne \varnothing$. Let
\[
V:= \{ E \in \mathcal{M}^s(2, 2h, h^2+2\chi) \text{ such that } h^0(E) \ne 0\}
\]
The following Proposition is a slight modification of Proposition 2.3 in \cite{SP} and Proposition 1 in \cite{PN}
\begin{proposition}\label{p3}
$\text{dim}(V) \le 2h^2 + 4\chi$ 
\end{proposition}
\begin{proof}
  Let $E \in V$. Then $E$ fits into an exact sequence 
  \[
   0 \longrightarrow \mathcal{O}_X \longrightarrow E \longrightarrow I_Z(2) \longrightarrow 0,
  \]
where $Z \in Hilb^{h^2 + 2\chi}(X)$ and $I_Z$ is the ideal sheaf corresponding to $Z$.
Let $\mathcal{N}(V)$ be the space of pairs
\[
 \mathcal{N}(V)= \{(E, s): E \in V, s \in \mathbb{P}(H^0(X, E))\}.
\]
Consider the following diagram
 \[
\xymatrix{
& \mathcal{N}(V) \ar[r]^{p_2} \ar[d]^{p_1} & \text{Hilb}^{h^2+2\chi}(X)\\
& V\\
}  
\]
Clearly $p_1$ is surjective. Thus $\text{dim }V \le \text{dim}(\mathcal{N}(V))$. On the other hand, 
\[
 \text{dim }p_2^{-1}(Z) = \text{dim }\mathbb{P}(\text{Ext}^1(I_{Z,X}(2), \mathcal{O}_{_X}))= h^1(I_Z(3))-1.
\]
Note that from the canonical exact sequence 
\[
 0 \longrightarrow I_Z(3) \longrightarrow \mathcal{O}_{_X}(3) \longrightarrow \mathcal{O}_Z(3) \longrightarrow 0,
\]
one can conclude that 
 $\mathbb{P}(\text{Ext}^1(I_Z(2), \mathcal{O}))$ is non-empty if and only if $ h^0(I_Z(3)) \ge 2h^2 -\chi +1$. \\

Let $\Delta_i = \{ Z \in \text{Hilb}^{h^2 + 2\chi}(X): h^0(I_Z(3)) =2h^2 -\chi + i\}$
and $\Delta = \cup_i\Delta_i$.  Thus $p_{_2}(\mathcal{N}(V)) \subset \Delta$.


Consider the incidence variety $T = \{(C, Z): Z \subset C\} \subset \mathbb{P}(H^0(\mathcal{O}_X(3))) \times \text{Hilb}^{h^2 + 2\chi}(X)$
and let $\pi_1, \pi_2$ are projections. Then dimension
of
$\pi_1^{-1}(C)$ is at most $h^2 + 2\chi$, so 
$3h^2 + \chi-1 + h^2 + 2\chi  \ge \text{dim }T \ge \text{dim }\pi_2^{-1}(\triangle_i) \ge \text{dim }\triangle_i +2h^2 - \chi + i-1$.
This implies that $\text{dim }\triangle_i$ is bounded above by $2h^2 + 4\chi-i$. 
Thus the $\text{dim}(p_2^{-1}(\triangle_i)) \le 2h^2 + 4\chi-i  + i= 2h^2 + 4\chi$, which is independent of $i$. \\
Thus the $\text{dim}(\mathcal{N}(V)) \le 2h^2 +4\chi$.
\end{proof}

\subsection {proof of Theorem \ref{T1}}
Note that by Proposition \ref{p2}, the space of stable  bundles appearing in the exact sequence, 
\[
0 \to \mathcal{O}_X(-1) \to E \to I_{\Gamma, X}(3) \to 0,
\]
where $\Gamma$ is a zero dimensional subscheme of length $4h^2 +\chi$, is non-empty. By \cite[corollary 3.5]{SIM1}, every component of such bundles has dimension 
at least, $3(4h^2 + 2 \chi) -h^0(X, \mathcal{O}_X(5)) -1 = 12h^2 + 6\chi -10h^2 -\chi -1= 2h^2 +5\chi -1$. On the other hand, by
Proposition \ref{p3} the dimension stable vector bundles with at least a non-zero section is at most $2h^2 +4\chi$. 
Thus a general element in the moduli space do not admit any section. 
Thus for a general element $E \in \mathcal{M}^s(2, 2h, h^2 + 2\chi)$, $E(1)$ satisfies the hypothesis of corollary \ref{C1}, hence Ulrich. 
\begin{remark}
The hypothesis of the theorem is not a strong hypothesis. There 
are interesting family of examples satisfying the hypothesis. For example any general quintic hypersurface in $\mathbb{P}^3$ certainly satisfies our hypothesis. 

\end{remark}
\begin{remark}\label{RZ}
From the proof of Theorem \ref{T1}, It is clear that a general element in every component of $\Sigma(X, 4h^2 + 2\chi)$ is Ulrich.
\end{remark}
\begin{remark}
Proof of Theorem \ref{T1}, may work under a weaker hypothesis that the canonical bundle is only ample (need not be very ample). In fact,  in this situation, there is an integer $m$ such that $K_X^{\otimes m}$ is very ample. Then considering  $h:= K_X^{\otimes m}$ and using the embedding of $X$ via $h$  probably one can show the existence of Ulrich bundle with respect to $h$. 
\end{remark}
\section{Ulrich Wildness of $X$}
In this section we will prove the Theorem \ref{T2}.\\
{\bf Proof of Theorem \ref{T2}}:\\
By remark \ref{RZ}, a general element in $\Sigma(X, 4h^2 + 2\chi)$ is Ulrich. Also every irreducible component of $\Sigma(X, 4h^2 + 2\chi)$ has dimension at least $2h^2 +5\chi -1 > 0$. Let $A$ and $B$ be two non-isomorphic stable Ulric bundles.  Thus there does not exist any non-zero morphism from $A$ to $B$. \\
Therefore, by \ref{T3}, it is enough to prove that $h^1(X, A \otimes B^*) \ge 3$.\\
Note that if $E$ and $F$ be rank 2 vector bundles on a smooth projective complex surface $X$. Then one has
\begin{equation*}
  c_2(E \otimes F) = c_1(E)^2 + c_1(F)^2 + 2 c_2(E) + 2 c_2(F) + 3 c_1(E)c_1(F).
\end{equation*}
Taking $F = B^*$ and $E = A$, we have,
$c_2(A \otimes B^*) = 8\chi$.  \\
Thus $\chi(A \otimes B^*)= -8\chi + 4\chi=-4\chi$.  Therefore, $h^1(X, A \otimes B^*) \ge 4\chi \ge 3$ and we are done.

\end{document}